 \documentclass{amsart} 
 \usepackage{times}
\usepackage{latexsym,amssymb,amsmath,hyperref,mathrsfs}
\usepackage{xcolor}
\usepackage{bm}
\newcommand{\R}{{\mathbb R}}
\newcommand{\N}{{\mathbb N}}

\newcommand{\rY}{{\mathrm Y}}
\newcommand{\rZ}{{\mathrm Z}}
\newcommand{\rX}{{\mathrm X}}

\def\loc{\mathrm{loc}}

\def\XXint#1#2#3{{\setbox0=\hbox{$#1{#2#3}{\int}$}
     \vcenter{\hbox{$#2#3$}}\kern-.5\wd0}}

\newtheorem{thm}{Theorem}[section]  
\newtheorem{lem}[thm]{Lemma}	       
\newtheorem{crlr}[thm]{Corollary}      
\newtheorem{prp}[thm]{Proposition}     
\newtheorem{defin}[thm]{Definition}    
	
\usepackage{setspace}

\numberwithin{equation}{section}

\begin{document}

\title{Nonlocal problems for Laplace equation in Bochner spaces}

\author[B.Bilalov, S.Sadigova, P.Salerno, L.Softova]{Bilal Bilalov, Sabina Sadigova, Pia Salerno, Lubomira Softova}

\address{B. Bilalov, Yildiz Technical University, Istanbul, Turkiye}
\address{Institute of Mathematics, The Ministry of Science and Education, Baku, Azerbaijan}
\address{Azerbaijan University of Architecture and Construction, Baku, Azerbaijan}
\address{Odlar Yurdu University, Baku, Azerbaijan}
\email{bilal.bilalov@yildiz.edu.tr, bilalov.bilal@gmail.com, ORCID: 0000-0003-0750-9339.}

\address{S. Sadigova, Institute of Mathematics, The Ministry of Science and Education, Baku, Azerbaijan}
\email{sadigova.sr@gmail.com, s\_sadigova@mail.ru, ORCID: 0000-0003-4654-0494.}

\address{P. Salerno, Department of Mathematics, University of Salerno, 84084, Fisciano (SA), Italy}
\email{psalerno@unisa.it, ORCID: 0009000624391159.} 

\address{L. Softova, Department of Mathematics, University of Salerno, 84084, Fisciano (SA), Italy} \email{lsoftova@unisa.it, ORCID: 0000000294989088.} 


\begin{abstract}
We study the Laplace equation posed in the semi-infinite strip $\Pi = I \times (0,\infty)$ with $I= (0,2\pi)$, and subject to nonlocal boundary conditions on $\partial \Pi$ in the trace sense. 
The analysis is carried out in the Bochner-Sobolev space $W^2_{p,1}(\Pi;\rX)$, associated with the Bochner space $L^{p,1}(\Pi;\rX)$, with $ p \in (1,\infty)$ and $\rX$ is a suitable Hilbert  space. To solve the problem, we employ a generalized spectral  method. In particular, we introduce the notion of $\otimes$-basis generated by tensor products and extend the classical scheme known from the scalar case to the present setting. 

Moreover, we prove that the system of root functions of the corresponding nonlocal spectral problem forms a $\otimes$-basis in $L^p(I;\rX)$. 
\end{abstract}
\maketitle
\allowdisplaybreaks
\section{Introduction}\label{sec1}

Problems arising in mechanics, mathematical physics, and pure mathematics naturally lead to the study of differential equations in more general functional frameworks, such as Morrey, Grand Lebesgue, and Orlicz spaces (see, e.g., \cite{B1,B2,B4,B7,BSaSf} and the references therein). These spaces provide a more flexible and robust setting, enabling a finer characterization of local regularity, growth conditions, and integrability properties of solutions. In particular, they are well suited for capturing nonstandard phenomena, including variable smoothness, non-uniform integrability, and anisotropic features that frequently arise in applications.

At the same time, many contemporary applications give rise to classes of equations that fall outside the scope of the classical theory (see, e.g., \cite{AlB,43,42,LR,Ma,Mo}). These include equations with discontinuous or merely measurable coefficients, nonlinear structures, and operators exhibiting nonstandard growth or degeneracy. Such challenges necessitate the development of new analytical approaches, as well as the extension and refinement of existing solvability results to broader and more realistic settings.

The present work is devoted to the generalization and further development of the aforementioned problems. In particular, we establish unique solvability results for a boundary value problem associated with the Laplace equation in an unbounded strip $\Pi \subset \mathbb{R}^2$. By imposing minimal regularity assumptions on the boundary data and adopting a non-harmonic analysis approach, we obtain solvability in Sobolev--Bochner spaces. Recall that Bochner integrability is defined with respect to a Banach space and extends the classical notion of the Lebesgue integral to Banach-valued measurable functions through the norm of $f$.

To this end, we introduce the basic notions of Bochner spaces and Unconditional Martingale Difference (UMD) spaces (see \cite{B99,DU,HNVW,W}).

In order to apply the generalized spectral method, we introduce the algebraic tensor product ($\otimes$-product) of two Banach spaces. Moreover, the $\otimes$-product structure on the Banach space $\rX$ allows us to extend the classical notions of basis and completeness to the notions of $\otimes$-basis and $\otimes$-completeness, respectively.

Since $\rX$ need not be a Hilbert space, we replace the notion of orthogonality with that of biorthogonality, more precisely, $\otimes$-biorthogonality (see \cite{B99,BHS,B4,B8,B0,HMVZ,Ma,Mo,S}).

Furthermore, we adapt the classical Fourier method of separation of variables to our framework, which enables us to prove the unique solvability of the modal boundary value problem \eqref{BVP} in the general class of Bochner--Sobolev spaces. This problem serves as a starting point for further investigations of the singular equation $ 
y^m u_{xx}+u_{yy}=0,$ $ m>-2,$ 
first studied by Moiseev in \cite{Mo}.

In this paper, we use the following  notation:
\begin{itemize}
\item $I := (0, 2\pi)$, $I_0 = (0,2\pi) \times \{0\} = \{(x,0) \ : \ x \in (0,2\pi) \}$. 
\item 
$J_0 := \{(0,y) \ : \ y \in (0,\infty)\}, \ 
J_{2\pi}:= \{ (2\pi,y) \ : \ y \in (0,\infty)\}.
$ 
 	\item $\Pi := (0,2\pi) \times (0,\infty)$, \ 
    $\Pi_\xi := I \times (0,\xi)$ \ for all $\xi > 0$. 
    
    \item 
    $J^\xi_0 := \{(0,y) \ : \ y \in (0,\xi)\}$,  
    $J^\xi_{2\pi}:= \{ (2\pi,y) \ : \ y \in (0,\xi)\}$.

    \item For each Banach space $\rX$, 
    $\|\cdot\|_{\rX}$ denotes its norm, and $\rX^*$ its dual space.
    
    \item $[\rX;\rY]$ denotes the Banach space of bounded linear operators from $\rX$ to $\rY$,
    \mbox{$[\rX;\rX]=[\rX]$.}
    
    \item $\alpha = (\alpha_1,\ldots,\alpha_n)$ is a multi-index with 
    $|\alpha| = \sum_{k=1}^n\alpha_k $  and 
    $$
    \partial^\alpha u(x) =
  \partial_{x_1}^{\alpha_1} \cdots\partial_{x_n}^{\alpha_n} u(x),\quad x\in \mathbb{R}^n.
    $$
    
    \item For any measurable set $\mathcal{S}$, $|\mathcal{S}|$ denotes the Lebesgue measure of $\mathcal{S}$,  and 
    $\overline{\mathcal{S}}$ denotes the closure of $\mathcal{S}$.

    \item $f|_\mathcal{S}$ denotes the restriction of $f$ to $\mathcal{S}$.
    
    \item $p'$ denotes the conjugate exponent of $p$, i.e.,
    $\frac{1}{p} + \frac{1}{p'} = 1$.

    \item For any domain $\Omega$ and integer $m \ge 1$, 
    $
    C_0^\infty(\Omega;\rX)$ -- the set of all infinitely differentiable $\rX$-valued functions with compact support in $\Omega$; 
    $
    C^m(\Omega;\rX)$ \\ $( \text{or} \ C^m(\overline{\Omega};\rX))
    $ -- the set of all $m$-th order differentiable $\rX$-valued functions on $\Omega$ (or $\overline{\Omega}$). 
        
    \item $\delta_{nk}$ denotes the Kronecker delta.
    
    \item $C$ denotes a positive constant which may change from line to line.
\end{itemize}

\section{Functional Analytic Framework}\label{sec2}

\subsection{Tensor Products, Bochner Spaces, and UMD Spaces}

Let $\rX$ and $\rY$ be Banach spaces. The algebraic \textit{tensor product} $\rX \otimes \rY$ is a vector space equipped with a bilinear map
$$
\otimes : \rX \times \rY \longrightarrow \rX \otimes \rY, \qquad
(x,y) \longmapsto x \otimes y,
$$
satisfying the following \textit{universal property}: for every vector space $\rZ$ and every bilinear map
$$
\tau : \rX \times \rY \longrightarrow \rZ,
$$
there exists a unique linear map
$$
\widetilde{\tau} : \rX \otimes \rY \longrightarrow \rZ
$$
such that
$$
\widetilde{\tau}(x \otimes y) = \tau(x,y),
\qquad x \in \rX,\quad y \in \rY.
$$

Moreover, the algebraic tensor product $\rX \otimes \rY$ is the vector space generated by the elementary tensors $x \otimes y$, with $x \in \rX$ and $y \in \rY$. Thus, every element $u \in \rX \otimes \rY$ can be written as
$$
u = \sum_{i=1}^{m} x_i \otimes y_i,
\qquad m \in \mathbb{N},
$$
for some $x_i \in \rX$ and $y_i \in \rY$.

Let $(\mathcal{S},\mathcal{A},\mu)$ be a measure space, where $\mathcal{A} \subseteq 2^{\mathcal{S}}$ is a $\sigma$-algebra and $\mu$ is a measure on $\mathcal{A}$. An $\rX$-valued function $f : \mathcal{S} \to \rX$ is called \textit{$\mu$-simple} if it can be represented in the form
$$
f = \sum_{k=1}^{n} \chi_{A_k} x_k,
\qquad n \in \mathbb{N},
$$
where, for each $1 \leq k \leq n$, $x_k \in \rX$, $A_k \in \mathcal{A}$, $\mu(A_k) < \infty$, and the sets $A_k$ are pairwise disjoint. Here, $\chi_{A_k}$ denotes the indicator function of the set $A_k$.

For an $\rX$-valued $\mu$-simple function $f$ as above, we define its
Bochner integral by
\begin{equation*}
\int_{\mathcal{S}} f\,d\mu := \sum_{k=1}^n\mu(A_k)x_k.
\end{equation*}

A function $f:\mathcal{S}\to\rX$ is said to be
\textit{strongly $\mu$-measurable} if there exists a sequence
$\{f_j\}_{j\in\mathbb{N}}$ of $\rX$-valued $\mu$-simple functions such that
$$
f_j(s)\to f(s)
\qquad\text{for $\mu$-almost every }s\in\mathcal{S}.
$$

The \textit{Bochner integral} of $f$ with respect to $\mu$ is defined by
$$
\int_{\mathcal{S}} f\,d\mu
:=\lim_{j\to\infty}
\int_{\mathcal{S}} f_j\,d\mu,
$$
where the limit is taken in $\rX$.

For simplicity, we identify the elementary tensor $f\otimes x$ with
the notation
\begin{equation}\label{eq-product}
xf:=f\otimes x,
\qquad x\in\rX,\quad f\in L^p(\mathcal{S}).
\end{equation}

The following classical result ensures the density of $\rX$-valued
simple functions in a Bochner space; see, e.g., \cite{DU,HNVW}.

\begin{lem}\label{lem:simple-density}
The set of $\mu$-simple $\rX$-valued functions is dense in
$L^p(\mathcal{S};\rX)$ for $1\leq p<\infty$.
In particular,
$$
\overline{L^p(\mathcal{S})\otimes\rX} =
L^p(\mathcal{S};\rX),
$$
where the closure is taken with respect to the Bochner norm.
\end{lem}

We also recall the notion of UMD spaces; see, e.g., \cite{HNVW,W}.

\begin{defin}
A Banach space $\rX$ is said to have the
\textit{unconditional martingale difference} (UMD) property if, for every
$p\in(1,\infty)$, there exists a constant
$\beta_{p,\rX}\geq 1$ such that, for every $\sigma$-finite measure space
$(\mathcal{S},\mathcal{A},\mu)$, every finite filtration
$$
\mathcal{F}_0\subseteq\mathcal{F}_1\subseteq\cdots\subseteq
\mathcal{F}_n\subseteq\mathcal{A},
$$
and every finite $\rX$-valued martingale
$\{f_k\}_{k=0}^n$ in $L^p(\mathcal{S};\rX)$ adapted to this filtration, one has
$$
\left\|
\sum_{k=1}^n \xi_k\,df_k
\right\|_{L^p(\mathcal{S};\rX)} \leq \beta_{p,\rX}
\left\| \sum_{k=1}^n\, df_k \right\|_{L^p(\mathcal{S};\rX)}
$$
for every choice of scalars $\xi_1,\ldots,\xi_n$ satisfying $|\xi_k|=1$, where
$$
df_k=f_k-f_{k-1}, \qquad k=1,\ldots,n.
$$
\end{defin}

\subsection{Completeness, biorthogonality, and basis}
\label{sub2.2}

Let $\vec{f}=\{f_k\}_{k\in\mathbb{N}}$ be a system of functions in $L^p(\mathcal{S})$. 
The $\otimes$-\textit{span} of $\vec f$ is defined by
$$
L_\otimes[\vec f] =
\left\{ g\in L^p(\mathcal{S};\rX):
g=\sum_{k\in F}x_kf_k
\text{ for some finite } F\subset\mathbb{N}
\text{ and } x_k\in\rX
\right\}.
$$

\begin{defin}\label{complete}
A system
$\vec f=\{f_k\}_{k\in\mathbb{N}}\subset L^p(\mathcal{S})$
is said to be \textit{$\otimes$-complete} in
$L^p(\mathcal{S};\rX)$ if
$$
\overline{L_\otimes[\vec f]} =
L^p(\mathcal{S};\rX),
$$
where the closure is taken with respect to the Bochner norm.
\end{defin}

\begin{defin}
A system of bounded linear operators
$$
\{t_n\}_{n\in\mathbb N} \subset [L^p(\mathcal{S};\rX);\rX]
$$
is said to be \textit{$\otimes$-biorthonormal} to $\vec f$ if
\begin{equation}\label{biortog}
t_n(xf_k)=\delta_{nk}x,
\qquad
\forall\,x\in\rX,\quad \forall\ n,k\in\mathbb N.
\end{equation}
\end{defin}

The following result relates the notions introduced above and extends
some previous results obtained in \cite{BHS,B4,B8} to the framework of Bochner and UMD spaces.
\begin{lem}\label{L1}
Every $\otimes$-complete system
$\vec{f}=\{f_k\}_{k\in\N}\subset L^p(\mathcal{S})$
admits at most one $\otimes$-biorthonormal  system.
\end{lem}

\begin{proof}
Assume that $\vec{f}$ admits two $\otimes$-biorthonormal  systems
$$
\{t_n^i\}_{n\in\N}
\subset [L^p(\mathcal{S};\rX);\rX], \qquad i=1,2,
$$
and set $ t_n:=t_n^1-t_n^2.$ 
Then, for every $x\in\rX$ and every $n,k\in\N$,
\begin{equation}\label{eq-tnk}
t_n(xf_k) =
t_n^1(xf_k)-t_n^2(xf_k) =
\delta_{nk}x-\delta_{nk}x = 0.
\end{equation}

Let $g\in L^p(\mathcal{S};\rX)$ be arbitrary. Since $\vec{f}$ is $\otimes$-complete, for every $\varepsilon>0$ there exist a finite set $F\subset \mathbb{N}$  and vectors $x_k\in \rX,$ $k\in F,$ such that 
$$
\left\| g-\sum_{k\in F}\, x_kf_k
\right\|_{L^p(\mathcal{S};\rX)} <
\frac{\varepsilon}  {\|t_n\|_{[L^p(\mathcal{S};\rX);\rX]}+1}.
$$
By the linearity of $t_n$ and \eqref{eq-tnk}
we have 
$$
t_n\left(\sum_{k\in F}x_kf_k\right) = \sum_{k\in F}t_n(x_kf_k) = 0. 
$$
Consequently,
\begin{align*}
\|t_n(g)\|_{\rX}
&=\left\|
t_n\left( g-\sum_{k\in F}\, x_kf_k \right) \right\|_{\rX}\\
&\leq \|t_n\|_{[L^p(\mathcal{S};\rX);\rX]}
\left\| g-\sum_{k\in F}\, x_kf_k
\right\|_{L^p(\mathcal{S};\rX)}\\
&< \frac{\|t_n\|_{[L^p(\mathcal{S};\rX);\rX]}}
{\|t_n\|_{[L^p(\mathcal{S};\rX);\rX]}+1}
\,\varepsilon
< \varepsilon.
\end{align*}
Since $\varepsilon>0$ is arbitrary, it follows that
$  t_n(g)=0.$ 
As $g\in L^p(\mathcal{S};\rX)$ was arbitrary, we conclude that
$  t_n\equiv0.$ 
Hence
$$
t_n^1=t_n^2 \qquad\text{for all }n\in\N,
$$
which proves the uniqueness.
\end{proof}

\begin{defin}\label{basis}
A system $\vec{f}=\{f_k\}_{k\in\N}\subset L^p(\mathcal{S})$ is said to
form a \textit{$\otimes$-basis} for $L^p(\mathcal{S};\rX)$ if, for every
$g\in L^p(\mathcal{S};\rX)$, there exists a unique sequence
$\{x_k\}_{k\in\N}\subset\rX$ such that
\[
g=\sum_{k=1}^{\infty}x_kf_k,
\]
where the series converges in $L^p(\mathcal{S};\rX)$.

Equivalently, every $g\in L^p(\mathcal{S};\rX)$ admits a unique
representation
\[
g=\sum_{k=1}^{\infty}f_k\otimes x_k,
\]
where the series converges in $L^p(\mathcal{S};\rX)$.
\end{defin}

The following criterion is analogous to the classical characterization
of a Schauder basis.

\begin{thm}\label{T1}
The system $\vec{f}=\{f_k\}_{k\in\N}\subset L^p(\mathcal{S})$ forms a
$\otimes$-basis for $L^p(\mathcal{S};\rX)$ if and only if the following
conditions are satisfied:
\begin{enumerate}
\item $\vec{f}$ is $\otimes$-complete in $L^p(\mathcal{S};\rX)$;

\item $\vec{f}$ admits a $\otimes$-biorthonormal  system
$$
\{t_n\}_{n\in\N} \subset [L^p(\mathcal{S};\rX);\rX]
$$
satisfying \eqref{biortog};

\item the sequence of projection operators
$$
\{P_m\}_{m\in\N}
\subset [L^p(\mathcal{S};\rX)]
$$
defined by
$$
P_m(g)
=\sum_{k=1}^{m}\, t_k(g)f_k,
\qquad
g\in L^p(\mathcal{S};\rX),
$$
is uniformly bounded, that is,
$$
\sup_{m\in\N}
\|P_m\|_{[L^p(\mathcal{S};\rX)]}<\infty.
$$
\end{enumerate}
\end{thm}

The proof follows the standard characterization of Schauder bases, together with Lemma~\ref{L1}.

\subsection{Bochner--Sobolev spaces and trace operators}

Let $\Omega\subset\mathbb{R}^n$, $n\geq1$, be an open connected set.  The \textit{Bochner--Sobolev space}
$W_p^m(\Omega;\rX)$, $1\leq p<\infty$, $m\in\mathbb{N}$, is defined by
$$
W_p^m(\Omega;\rX)
= \left\{
u\in L^p(\Omega;\rX):
\partial^\alpha u\in L^p(\Omega;\rX)
\text{ for all } \ |\alpha|\leq m
\right\},
$$
where $\partial^\alpha u$ denotes the weak derivative of $u$ of order $|\alpha|$. 
It is a Banach space equipped with the norm
$$
\|u\|_{W_p^m(\Omega;\rX)} =
\sum_{|\alpha|\leq m}
\|\partial^\alpha u\|_{L^p(\Omega;\rX)},
$$
called the \textit{Bochner--Sobolev norm}.

A function $f:\Omega\to\rX$ is said to be
\textit{locally Bochner integrable}, denoted by
$f\in L^1_{\mathrm{loc}}(\Omega;\rX)$, if
$ f\in L^1(K;\rX)$ 
for every compact set $K\subset\Omega$.

We also recall the following Bochner-valued version of the fundamental
theorem of calculus; see \cite[Lemma 2.5.8]{HNVW}.

\begin{lem}
Let $g\in L^1_{\loc}(\mathbb{R};\rX)$ and let $a\in\mathbb{R}$. Define
$f:\mathbb{R}\to\rX$ by
$$
f(x)=\int_a^x g(s)\,ds, \qquad x\in\mathbb{R},
$$
where the integral is understood in the Bochner sense.
Then $f\in W^1_{1,\loc}(\mathbb{R};\rX)$ and
$$
\partial f=f'=g
\qquad\text{ for almost every } x\in \mathbb{R}.
$$
\end{lem}

We say that a domain $\Omega\subset\mathbb{R}^n$, $n\geq2$, satisfies the
\textit{segment condition} if, for every $z\in\partial\Omega$, there exist
an open neighbourhood $U_z$ of $z$ and a vector
$h_z\in\mathbb{R}^n\setminus\{0\},$ depending on $z$, such that
$$
\xi+th_z\in\Omega,
\qquad
\forall\,\xi\in U_z\cap\overline{\Omega},
\quad \forall\,t\in(0,1).
$$

The following result gives a Bochner-valued extension of the classical
density result for Sobolev spaces; see \cite{HNVW}. 
In particular, the
segment condition is satisfied by domains with Lipschitz boundary.

\begin{prp}
Let $\Omega$ satisfy the segment condition, let $m\in\mathbb{N}$, and let
$1\leq p<\infty$. Then
$$
\left\{f|_\Omega: \ f\in C_0^\infty(\mathbb{R}^n;\rX)
\right\}
$$
is dense in $W_p^m(\Omega;\rX).$ 
\end{prp}

Let $S\subset\partial\Omega$ be an $(n-1)$-dimensional surface. We denote
by $C^\infty_{0;S}(\Omega;\rX)$ the set of all infinitely differentiable
$\rX$-valued functions on $\Omega$ that vanish in a neighborhood of $S$, that is,
$$
C^\infty_{0;S}(\Omega;\rX)
=
\left\{
u\in C^\infty(\Omega;\rX):
\exists \ U\subset\mathbb{R}^n
\text{ with } S\subset U
\text{ such that }
u|_{U\cap\Omega}=0 \right\}.
$$

The closure of $C^\infty_{0;S}(\Omega;\rX)$ in
$W^1_p(\Omega;\rX)$ is denoted by
$$
\mathring{W}^{1}_{p;S}(\Omega;\rX)
:= \overline{C^\infty_{0;S}(\Omega;\rX)}^{\,W^1_p(\Omega;\rX)}.
$$

We now introduce the notion of a trace operator and the corresponding
trace space. To this end, we use the following theorem, proved in
\cite{Kr1,Kr2}.

\begin{thm}\cite[Theorem 4.14]{Kr1}
Let $\Omega\subset\R^n$, $n\geq2$, be a  domain with uniform Lipschitz
boundary, and let $1\leq p<\infty$. Then there exists a linear and continuous  operator
$$
\Gamma\in
[W^1_p(\Omega;\rX);L^p(\partial\Omega;\rX)]
$$
such that
$$
\Gamma u=u|_{\partial\Omega}
\qquad
\forall\,u\in
W^1_p(\Omega;\rX)\cap C(\overline{\Omega};\rX).
$$

Moreover, for every $u\in W^1_p(\Omega;\rX)$,
$$
u\in\mathring W^1_p(\Omega;\rX) \quad\Longleftrightarrow\quad
\Gamma u=0.
$$
\end{thm}

As a consequence, we obtain the following corollary.

\begin{crlr}\label{cor2.2}
Let $\Omega\subset\R^n$, $n\geq2$, be a  domain with
Lipschitz boundary, and let 
$ S\subset\partial\Omega$ be a relatively open
Lipschitz portion of $\partial\Omega$. 
Let $1\leq p<\infty$. Then there
exists a bounded linear operator
$$
\Gamma_S\in [W^1_p(\Omega;\rX);L^p(S;\rX)]
$$
such that
$$
\Gamma_Su=u|_S,
\qquad \forall\ u\in W^1_p(\Omega;\rX)\cap C(\overline{\Omega};\rX).
$$
Moreover,
$$
u\in\mathring W^1_{p;S}(\Omega;\rX)
\quad\Longleftrightarrow\quad
\Gamma_S u=0.
$$
\end{crlr}

\begin{proof}
Let $\Gamma$ be the trace operator given by the preceding theorem. Define
$$
\Gamma_Su := (\Gamma u)|_S, \qquad u\in W^1_p(\Omega;\rX).
$$
Since the restriction operator
$$
L^p(\partial\Omega;\rX)\longrightarrow L^p(S;\rX)
$$
is bounded, $\Gamma_S$ is a bounded linear operator from
$W^1_p(\Omega;\rX)$ into $L^p(S;\rX)$. Moreover, if
$$
u\in W^1_p(\Omega;\rX)\cap C(\overline{\Omega};\rX),
$$
then, by the defining property of $\Gamma$,
$
\Gamma_Su=(\Gamma u)|_S=u|_S.
$

It remains to prove the characterization of the kernel of $\Gamma_S$.
Suppose first that
$
u\in\mathring{W}^{1}_{p;S}(\Omega;\rX). $
Then there exists a sequence
$\{u_j\}_{j\in\mathbb{N}}\subset C^\infty_{0;S}(\Omega;\rX)$ such that
$$
u_j\longrightarrow u
\qquad\text{in } \ W^1_p(\Omega;\rX).
$$
Since each $u_j$ vanishes in a neighborhood of $S$, we have
$\Gamma_Su_j=0,$ for all $ j\in \N.$
The continuity of $\Gamma_S$ therefore gives
$$
\Gamma_Su=
\lim_{j\to\infty}\Gamma_Su_j =0.
$$

Conversely, let $u\in W^1_p(\Omega;\rX)$ and 
$
\Gamma_Su=0.$ 
Since $S$ is a relatively open Lipschitz portion of $\partial\Omega$,
for every $z\in S$ there exists a neighborhood $U_z$ of $z$ and a
Lipschitz change of variables which maps
$\Omega\cap U_z$ onto a portion of a half-space and maps
$S\cap U_z$ onto a portion of its boundary. By a finite covering
argument, it is therefore enough to consider the local situation in a
Lipschitz half-ball.

In such a local coordinate system, the condition
$ 
\Gamma_Su=0$ 
means that the trace of $u$ on the flat part of the boundary is zero.
The standard zero-trace approximation result in a Lipschitz half-ball
then yields a sequence $  v_j\in C^\infty$ 
which vanishes in a neighborhood of the flat boundary and satisfies $ 
v_j\to u $ in $W^1_p $ 
on the corresponding local portion of $\Omega$.

Choose a partition of unity
$\{\varphi_j\}_{j=1}^N$ subordinate to the resulting finite covering
of a neighborhood of $S$, together with a function $\varphi_0$ whose
support is separated from $S$, such that $ 
\sum_{j=0}^N\varphi_j=1 $ 
on a neighborhood of $\overline{\Omega}$. 
Applying the preceding local
approximation to each $\varphi_j u$, $j=1,\ldots,N$, and using ordinary
mollification for $\varphi_0u$, we obtain a sequence
$ 
w_k\in C^\infty_{0;S}(\Omega;\rX)$ 
such that
$$
w_k\longrightarrow u
\qquad\text{in } \ W^1_p(\Omega;\rX).
$$
Hence
$$
u\in
\overline{C^\infty_{0;S}(\Omega;\rX)}
^{\,W^1_p(\Omega;\rX)}
=
\mathring{W}^{\,1}_{p;S}(\Omega;\rX).
$$
Therefore $ 
\ker\Gamma_S
\subset \mathring{W}^{\,1}_{p;S}(\Omega;\rX).$ 
Combining the two inclusions gives
$$
u\in\mathring{W}^{\,1}_{p;S}(\Omega;\rX)
\quad\Longleftrightarrow\quad
\Gamma_Su=0.
$$

\end{proof}

\subsection{Bochner--Sobolev spaces for functions defined on an interval}

Let $n=1$ and let $I=(0,2\pi)$. We consider the Bochner space
$L^p(I;\rX)$, endowed with the norm
\begin{equation}\label{eq-Boch-norm}
\|f\|_{L^p(I;\rX)}
=
\left(
\int_I\|f(x)\|_{\rX}^p\,dx
\right)^{1/p},
\qquad 1\leq p<\infty,
\end{equation}
and the corresponding Bochner--Sobolev space
$W_p^2(I;\rX)$, endowed with the norm
\begin{equation*}
\|f\|_{W_p^2(I;\rX)}
= \|f\|_{L^p(I;\rX)} + \|f'\|_{L^p(I;\rX)}+ \|f''\|_{L^p(I;\rX)}.
\end{equation*}


Consider the trigonometric system
$\{e^{inx}\}_{n\in\mathbb{Z}}$. An
\textit{$\rX$-valued trigonometric polynomial} is a function
$P:I\to\rX$ of the form
\[
P(x)=\sum_{k=-N}^{N}a_k e^{ikx},
\qquad x\in I,
\]
for some $N\in\mathbb{N}_0$ and coefficients $a_k\in\rX$.
We denote by $\mathcal{P}(\rX)$ the collection of all such
polynomials.

The following result is the vector-valued extension of the classical
density theorem for trigonometric polynomials (cf. \cite{HNVW}).

\begin{prp}\label{2.1}
Let $\rX$ be a Banach space. Then $\mathcal{P}(\rX)$ is dense in
$L^p(I;\rX)$ for every $1\leq p<\infty$.
\end{prp}

Moreover, we recall the following characterization of the boundedness of the Riesz projection on vector-valued $L^p$ spaces. In particular, for $1<p<\infty$, the Riesz projection is bounded on $L^p(I;\rX)$ whenever $\rX$ is a UMD space; see, for instance, \cite{B8,B0,S}, where the corresponding result is established for exponential Fourier series.

\begin{thm}[\cite{B8}]\label{T2.2}
Let $\rX$ be a UMD space and let $p\in(1,\infty)$. Then the exponential
system $
\{e^{inx}\}_{n\in\mathbb Z} $
forms an $\otimes$-basis of $L^p(I;\rX)$. Moreover, for every
$m\in\mathbb Z$, the Riesz projections
$$
R_m^+f(x) = \sum_{n=m}^{\infty}\widehat f(n)e^{inx},
\qquad R_m^-f(x) =\sum_{n=-\infty}^{m-1}\widehat f(n)e^{inx},
$$
are bounded linear operators on $L^p(I;\rX)$, where
$$
\widehat f(n) =
\frac{1}{2\pi} \int_0^{2\pi}f(x)e^{-inx}\,dx,
\qquad n\in\mathbb Z.
$$
\end{thm}

The symmetric Fourier partial sums associated with the exponential system can be written in terms of sine and cosine functions. Hence, Theorem~\ref{T2.2} yields the following result.

\begin{crlr}\label{cor-2.2}
Let $\rX$ be a UMD space and let $p\in(1,\infty)$. Then the ordered
trigonometric system
$$
\mathcal{T} = \{1,\cos x,\sin x,\cos(2x),\sin(2x),\ldots\}
$$
forms an $\otimes$-basis of  $L^p(I;\rX)$. More precisely, for every
$f\in L^p(I;\rX)$,
$ f =
\lim_{n\to\infty}S_nf $
in $L^p(I;\rX)$, where
$$
S_nf(x) =
\ell_0^c(f) +
\sum_{k=1}^{n}
\bigl(
\ell_k^c(f)\cos(kx) +
\ell_k^s(f)\sin(kx)
\bigr),
$$
and
\begin{align*}
\ell_0^c(f)& =
\frac{1}{2\pi}
\int_I f(x)\,dx, \quad 
\ell_k^c(f) =  \frac{1}{\pi}
\int_I f(x)\cos(kx)\,dx,\\
\ell_k^s(f)& =
\frac{1}{\pi}
\int_I f(x)\sin(kx)\,dx,
\qquad k\in\mathbb N.
\end{align*}
Here all integrals are understood in the Bochner sense.
\end{crlr}

Let us note that, in this case, the corresponding biorthonormal  system
of $\mathcal{T}$ is obtained from $\mathcal{T}$ by suitable
normalization. More precisely,
$$
\mathcal{T}^* =
\left\{\frac{1}{2\pi}, \ 
\frac{1}{\pi}\cos(nx), \ 
\frac{1}{\pi}\sin(nx)
\right\}_{n\in\mathbb{N}},
$$
and the Fourier coefficients are given by the pairings with the
elements of $\mathcal{T}^*$.

We introduce the following notion.

\begin{defin}
The $\otimes$-basis $\mathcal{T}$ in $L^p(I;\rX)$, with
$p\in(1,\infty)$, is said to have the
$\otimes$-Riesz property if the cosine and sine partial-sum operators
$$
S_n^c f(x) :=
\sum_{k=0}^n \ell_k^c(f)\cos(kx),
\qquad
S_n^s f(x):=
\sum_{k=1}^n \ell_k^s(f)\sin(kx),
\qquad n\in\mathbb{N},
$$
are uniformly bounded on
$L^p(I;\rX)$, that is,
$$
\sup_{n\in\mathbb{N}}
\left(
\|S_n^c\|_{[L^p(I;\rX)]}
+\|S_n^s\|_{[L^p(I;\rX)]}
\right) <\infty.
$$
\end{defin}

For further details on these and related topics, we refer  to the monographs \cite{HNVW} and the papers \cite{B99,BHS,B8,B9,B0,S}. The following result holds.
\begin{prp}\label{propT}
Let $\rX$ be a UMD space. Then the ordered trigonometric system
$$
\mathcal{T} =
\{1,\cos x,\sin x,\cos(2x),\sin(2x),\ldots\}
$$
forms a $\otimes$-basis of $L^p(I;\rX)$ for every
$p\in(1,\infty)$ and possesses the $\otimes$-Riesz property.
\end{prp}

\begin{proof}
By Theorem~\ref{T2.2}, the exponential system
$\{e^{inx}\}_{n\in\mathbb Z}$ is a $\otimes$-basis in
$L^p(I;\rX)$, and its symmetric Fourier partial sums
$$
P_nf(x) =\sum_{k=-n}^{n}\widehat f(k)e^{ikx}
$$
converge to $f$ in $L^p(I;\rX)$.
For $k\geq1$, we have
$$
e^{ikx}=\cos(kx)+i\sin(kx),
\qquad e^{-ikx}=\cos(kx)-i\sin(kx),
$$
and
$$
\widehat f(k) =
\frac12\left(\ell_k^c(f)-i\ell_k^s(f)\right),
\qquad
\widehat f(-k) =
\frac12\left(\ell_k^c(f)+i\ell_k^s(f)\right).
$$

Consequently,
$$
\widehat f(k)e^{ikx} +
\widehat f(-k)e^{-ikx}
= \ell_k^c(f)\cos(kx) +
\ell_k^s(f)\sin(kx).
$$
Hence
$$
P_nf(x) =
\ell_0^c(f) +
\sum_{k=1}^{n}
\bigl(
\ell_k^c(f)\cos(kx) +
\ell_k^s(f)\sin(kx)
\bigr) =
S_nf(x).
$$
Therefore $ 
S_nf\longrightarrow f
\qquad\text{in }L^p(I;\rX). $ 
Thus $\mathcal{T}$ forms a $\otimes$-basis of $L^p(I;\rX)$.

It remains to prove the $\otimes$-Riesz property. Suppose that $f$ is
extended to a $2\pi$-periodic function. For
$$
f=f^++f^-,\qquad
f^\pm(x)=\frac{f(x)\pm f(-x)}{2}.
$$
Then $f^+$ is even and $f^-$ is odd. Hence
$ \ell_k^c(f^-)=0,$ $ \ell_k^s(f^+)=0, $
and therefore
$$
S_n^cf=S_nf^+,
\qquad S_n^sf=S_nf^-.
$$
Since $S_nf\to f$ in $L^p(I;\rX)$, it follows that
$$
S_n^cf\longrightarrow f^+,
\qquad
S_n^sf\longrightarrow f^-
$$
in $L^p(I;\rX).$  Consequently,
$$
\sup_{n\in\mathbb N}\|S_n^cf\|_{L^p(I;\rX)}<\infty,
\qquad
\sup_{n\in\mathbb N}\|S_n^sf\|_{L^p(I;\rX)}<\infty.
$$
By the Banach--Steinhaus theorem, the families
$\{S_n^c\}_{n\in\mathbb N}$ and 
$\{S_n^s\}_{n\in\mathbb N}$ are uniformly
bounded on $L^p(I;\rX)$. Hence $\mathcal{T}$ possesses the $\otimes$-Riesz property.
\end{proof}

\subsection{Bochner--Sobolev spaces for functions defined in an infinite strip}

We now adapt the above definitions to functions defined on the unbounded
strip
\[
\Pi=I\times(0,\infty)\subset\mathbb{R}^2.
\]

In what follows, we consider functions possessing a
\textit{mixed regularity} property. To this end, we introduce the
\textit{mixed-norm Bochner space} $L^{p,1}(\Pi;\rX)$ defined by
\begin{equation*}
\|u\|_{L^{p,1}(\Pi;\rX)}
=
\int_0^\infty
\left(
\int_0^{2\pi}
\|u(x,y)\|_{\rX}^p\,dx
\right)^{1/p}dy
=
\int_0^\infty
\|u(\cdot,y)\|_{L^p(I;\rX)}\,dy.
\end{equation*}

The \emph{mixed-norm Bochner--Sobolev space}
$W^m_{p,1}(\Pi;\rX)$ is defined as the space of all functions
$u\in L^{p,1}(\Pi;\rX)$ whose weak derivatives up to order $m$
exist and belong to $L^{p,1}(\Pi;\rX)$. It is equipped with the norm
\begin{equation*}
\|u\|_{W^m_{p,1}(\Pi;\rX)}
=
\sum_{|\alpha|\leq m}
\|\partial^\alpha u\|_{L^{p,1}(\Pi;\rX)}.
\end{equation*}


Let $u\in W^1_{p,1}(\Pi;\rX)$. Then, for every $\xi>0$, the restriction
of $u$ to $\Pi_\xi:= I\times (0,\xi)$ belongs to $W^1_{p,1}(\Pi_\xi;\rX)$. Hence, by the
mixed-norm trace theorem, for every $\xi>0$ there exists a bounded linear
operator
$$
\Gamma_{J_0^\xi} \in[W^1_{p,1}(\Pi_\xi;\rX);L^1(J_0^\xi;\rX)]
$$
such that $  \Gamma_{J_0^\xi}u=u|_{J_0^\xi}$ 
for every
$$
u\in W^1_{p,1}(\Pi_\xi;\rX)
\cap C(\overline{\Pi}_\xi;\rX).
$$

From the uniqueness of the trace operator, the compatibility condition
\begin{equation*}
\left.
\big(\Gamma_{J_0^{\xi_2}}u\big)
\right|_{J_0^{\xi_1}}
=
\Gamma_{J_0^{\xi_1}}u,
\qquad
\text{a.e. on }J_0^{\xi_1},
\qquad
0<\xi_1<\xi_2.
\end{equation*}
Consequently, there exists a unique global trace operator
\begin{align*}
\Gamma_{J_0}:
W^1_{p,1}(\Pi;\rX)
&\longrightarrow L^1_{\mathrm{loc}}(J_0;\rX),\\
u&\longmapsto\Gamma_{J_0}u,
\end{align*}
such that, for every $\xi>0$,
\begin{equation*}
\left.
\Gamma_{J_0}u
\right|_{J_0^\xi}
=
\Gamma_{J_0^\xi}u
\qquad
\text{in }L^1(J_0^\xi;\rX).
\end{equation*}
In particular, $ 
\Gamma_{J_0}u=u|_{J_0}$ 
 on $J_0$ 
for every
$$
u\in W^1_{p,1}(\Pi;\rX)
\cap C(\overline{\Pi};\rX).
$$
Moreover, for every $\xi>0$, the restriction
$$
\left.\Gamma_{J_0}\right|_{J_0^\xi}
: W^1_{p,1}(\Pi;\rX)
\longrightarrow L^1(J_0^\xi;\rX)
$$
is bounded.

Analogously, one defines the trace operator $\Gamma_{J_{2\pi}}$
corresponding to $J_{2\pi}\subset\partial\Pi$.
For simplicity, we set $ 
J_1:=J_0,$ $
J_2:=J_{2\pi}.$ 

\begin{prp}
There exist trace operators
\begin{align*}
\Gamma_{J_k}:
W^1_{p,1}(\Pi;\rX)
&\longrightarrow L^1_{\mathrm{loc}}(J_k;\rX),
\\
u&\longmapsto\Gamma_{J_k}u,
\qquad k=1,2,
\end{align*}
such that $ 
\Gamma_{J_k}u=u|_{J_k}$  a.e. 
 on $J_k$ for every
$ 
u\in W^1_{p,1}(\Pi;\rX)\cap C(\overline{\Pi};\rX).$ 
Moreover, for every $\xi>0$,
$$
\Gamma_{J_k}u
\big|_{J_k^\xi} =
\Gamma_{J_k^\xi}(u|_{\Pi_\xi}) \qquad
\text{ in } \ L^1(J_k^\xi;\rX),
\qquad k=1,2.
$$
\end{prp}

\section{Generalized spectral method}

In order to study the existence of solutions to problem \eqref{BVP}, we
adapt the classical \textit{Fourier method}, based on the separation of
variables for the Laplace equation. Representing a  solution in the form $ u(x,y)=\varphi(x)\psi(y),$ 
leads  to  two second-order linear ordinary differential equations.

Since we are working in Bochner spaces endowed with an $\otimes$-product
structure, the classical Fourier method must be modified to fit our
framework.

The first generalization concerns the spectral  problem associated
with the $x$-dependence. We consider the following \textit{spectral problem}:
\begin{equation}\label{SPE}
\begin{cases}
\varphi''(x)+\lambda\varphi(x)=0,
& x\in (0,2\pi),\\
\varphi(0)=\varphi(2\pi),\\
\varphi'(0)=0.
\end{cases}
\end{equation}
The   eigenvalues and  corresponding eigenfunctions are 
\begin{equation}\label{eq-eigen}
\lambda_n=n^2, \quad  \varphi^c_n(x)=\cos(nx),\quad n=0,1,2,\ldots.
\end{equation}

Since the system of eigenfunctions is not $\otimes$-complete in
$L^p(I;\rX)$, we additionally need to determine the corresponding
\textit{associated functions}, which we denote by
$\{\varphi_n^s\}_{n\in\mathbb N}$.


We consider the following system of root functions:
\begin{equation}\label{eq-system}
\varphi_0^c=1, \ 
\varphi_n^c(x)=\cos(nx), \ 
\varphi_n^s(x)=-\frac{x}{2n}\sin(nx), \quad 
n\in\mathbb N.
\end{equation}

%

The system \eqref{eq-system2} is biorthonormal  to the system
\begin{equation}\label{eq-system2}
v_0^c(x)=\frac{2\pi-x}{2\pi^2}, \quad  
v_n^c(x)=\frac{2\pi-x}{\pi^2}\cos(nx),\quad
v_n^s(x)=-\frac{2n}{\pi^2}\sin(nx).
\end{equation}
More precisely,  
\begin{equation}\label{eq-normal}
\begin{split}
&\int_0^{2\pi}v_0^c(x)\varphi_0^c(x)\,dx=1,\\
&\int_0^{2\pi}v_n^c(x)\varphi_m^c(x)\,dx=\delta_{nm}, \quad
\int_0^{2\pi}v_n^s(x)\varphi_m^s(x)\,dx=\delta_{nm}.
\end{split}
\end{equation}

Hence, with the natural ordering of the root functions, the two
systems are biorthonormal. By the $\otimes$-basis theorem established in
Subsection~\ref{sub2.2}, it follows that \eqref{eq-system} forms an
$\otimes$-basis of $L^p(I;\rX)$, and that its unique
$\otimes$-biorthonormal  system is given by \eqref{eq-system2}.

\begin{thm}\label{T3.4}
Let $\rX$ be a UMD  Hilbert space and let $p\in(1,\infty)$. Then the system
$ \{\varphi_0^c, \  \varphi_n^c, \ \varphi_n^s\}_{n\in\mathbb N},
$
defined by \eqref{eq-system}, is $\otimes$-complete in
$L^p(I;\rX)$ and admits a unique $\otimes$-biorthonormal system
$
\{v_0^c, \  v_n^c, \ v_n^s\}_{n\in\mathbb N}.$
Moreover, the partial-sum operators
$$
P_n f =
t_0^c(f)\varphi_0^c +
\sum_{k=1}^n
\left[t_k^c(f)\varphi_k^c + t_k^s(f)\varphi_k^s
\right]
$$
where 
\begin{equation}\label{eq-t}
\begin{split}
t_0^c(f) &= \int_0^{2\pi}v_0^c(x)f(x)\,dx, \qquad 
t_n^c(f) =\int_0^{2\pi}v_n^c(x)f(x)\,dx,\\
t_n^s(f)&=\int_0^{2\pi}v_n^s(x)f(x)\,dx, \quad n\in \mathbb{N},
\end{split}
\end{equation}
with   $v_0^c$, $v_n^c$, and $v_n^s$  defined by \eqref{eq-system2},
are uniformly bounded on $L^p(I;\rX)$. More precisely, there exists a constant
$C>0$ independent of $n$ and $f$, such that
$$
\|P_n f\|_{L^p(I;\rX)} \leq C\|f\|_{L^p(I;\rX)}
$$
for every $f\in L^p(I;\rX)$ and every $n\in\mathbb N$.

All integrals are understood in the Bochner sense.

\end{thm}

\begin{proof}
First, by construction, we have 
$ t_0^c, \, t_n^c, \, t_n^s \in [ L^p(I;\rX);\rX] $
and  for
$f\in L^p(I;\rX)$, we consider the expansion coefficients defined by \eqref{eq-t}. 
 
We next prove $\otimes$-completeness. By \cite[Theorem 2.5]{B4},
the system  $ 
\{\varphi_0^c, \, \varphi_n^c, \, \varphi_n^s\}_{n\in\mathbb N} $ 
is a basis of the scalar space $L^p(I)$. In particular, for every
$f\in L^p(I)$,
$$
f= \lim_{N\to\infty}
\left( t^c_0(f)\varphi_0^c+ \sum_{n=1}^N
\bigl(
t_n^c(f)\varphi_n^c+t_n^s(f)\varphi_n^s
\bigr)\right)
$$
in $L^p(I)$.
Let $z\in \rX.$
Multiplying  the above identity by $z$ gives
$$
fz= \lim_{N\to\infty}
\left( t^c_0(f)z\varphi_0^c+
\sum_{n=1}^N \bigl( t_n^c (f)z\varphi_n^c+
t_n^s(f) z\varphi_n^s \bigr)\right)
$$
in $L^p(I;\rX)$. 
Thus every elementary tensor
$f\otimes z=fz$ belongs to the closure of the $\otimes$-span of the
system. Since the algebraic tensor product
$L^p(I)\otimes\rX$ is dense in $L^p(I;\rX)$, it follows that
$$
\overline{
L_\otimes\left[ \{\varphi_0^c\} \cup
\{\varphi_n^c,\varphi_n^s\}_{n\in\mathbb N}
\right]}=L^p(I;\rX).
$$
Hence the system is $\otimes$-complete.

It remains to prove the uniform boundedness of the associated
projection operators. Define
$$
P_n^c(f) =
\sum_{k=0}^n t_k^c(f)\varphi_k^c,
\qquad
P_n^s(f)= \sum_{k=1}^n t_k^s(f)\varphi_k^s.
$$
Then $P_n=P_n^c+P_n^s$.

For the cosine part, set $ g(x)=(2\pi-x)f(x).$ 
Then 
$$
t_0^c(f) =
\frac{1}{\pi}\ell_0^c(g), \quad  t_k^c(f) = \frac{1}{\pi}
\ell_k^c(g).
$$
Therefore,
$$
P_n^c(f) = \frac1\pi
\sum_{k=0}^n \ell_k^c(g)\cos(kx) = \frac1\pi S_n^c g.
$$
By Proposition~\ref{propT}, the cosine partial-sum operators are
uniformly bounded on $L^p(I;\rX)$. Hence
$$
\|P_n^c(f)\|_{L^p(I;\rX)} \leq C \|g\|_{L^p(I;\rX)}\leq C  \|f\|_{L^p(I;\rX)}.
$$

For the sine part, from the definition of $v_k^s$ we have
$$
t_k^s(f) =
-\frac{2k}{\pi^2}
\int_0^{2\pi}f(x)\sin(kx)\,dx = -\frac{2k}{\pi}\ell_k^s(f).
$$
Since $ 
\varphi_k^s(x) =
-\frac{x}{2k}\sin(kx),$ 
it follows that
$$
t_k^s(f)\varphi_k^s(x) = \frac{x}{\pi} \ell_k^s(f)\sin(kx).
$$
Consequently,
$$
P_n^s(f)(x) =
\frac{x}{\pi} \sum_{k=1}^n \ell_k^s(f)\sin(kx).
$$
Again, by Proposition~\ref{propT}, the sine partial-sum operators are
uniformly bounded on $L^p(I;\rX)$. Since multiplication by $x$ is a
bounded operator on $L^p(I;\rX)$ and 
$$
\|P_n^s(f)\|_{L^p(I;\rX)}
\leq
C  \left\| \sum_{k=1}^n
\ell_k^s(f)\sin(kx) \right\|_{L^p(I;\rX)}
\leq C\|f\|_{L^p(I;\rX)},
$$
where $C$ is independent of $n$ and $f$. Finally
$$
\|P_n(f)\|_{L^p(I;\rX)}
\leq C  \|f\|_{L^p(I;\rX)}.
$$
\end{proof}

\section{Main results}

Consider the following $\rX$-valued nonlocal boundary value problem:
\begin{equation}\label{BVP}
\begin{cases}
\Delta u(x,y) = 0, &  (x,y)\in \Pi,\\
\Gamma_I u = f(x), & x\in (0,2\pi)\\
\Gamma_{J_0}u = \Gamma_{J_{2\pi}}u,&  y\in (0,\infty)\\
\Gamma_{J_0}(\partial_xu)=0,& y\in (0,\infty)
\end{cases}
\end{equation}
where $f\in W^2_p(I;\rX)$ is a given function.

By a solution of problem \eqref{BVP}, we mean a function
$  u\in W^2_{p,1}(\Pi;\rX) $ 
such that  Laplace equation holds almost everywhere in $\Pi$ and
the boundary conditions in \eqref{BVP} are satisfied in the trace sense.
\begin{thm}[Uniqueness]\label{P3.1}
Let $f\in W^2_p(I;\rX)$, with $1<p<\infty$. If problem
\eqref{BVP} admits a solution in $W_{p,1}^2(\Pi;\rX)$, then the   solution
is unique.
\end{thm}

\begin{proof}
Let $u_1,u_2\in W_{p,1}^2(\Pi;\rX)$ be two solutions of
\eqref{BVP} corresponding to the same boundary datum $f$. Set
$ w=u_1-u_2.$ 
Then $w\in W_{p,1}^2(\Pi;\rX)$ and satisfies the homogeneous problem
\begin{equation}\label{E3.3}
\begin{cases}
\Delta w(x,y)=0, & \text{in }\Pi,\\
\Gamma_I w=0, & x\in I\\
\Gamma_{J_0}w=\Gamma_{J_{2\pi}}w, & y\in (0,\infty)\\
\Gamma_{J_0}(\partial_xw)=0,& y\in (0,\infty).
\end{cases}
\end{equation}
Indeed, the first boundary condition follows from
$$
\Gamma_I w =
\Gamma_I u_1-\Gamma_I u_2 =0.
$$
The remaining boundary conditions follow by linearity of the trace operators.

It therefore suffices to prove that every solution of
\eqref{E3.3} is identically zero.
Let $v\in\rX^*$ be arbitrary and define the scalar-valued function
$ w_v(x,y):=v(w(x,y)). $ 
Since $v$ is a bounded linear functional on $\rX$, we have
$ 
w_v\in W_{p,1}^2(\Pi)$ 
and, for every multi-index $\alpha$ with $|\alpha|\leq  2$,
$ \partial^\alpha w_v =
v(\partial^\alpha w) $ 
in the weak sense. Consequently,
$$
\Delta w_v =
v(\Delta w) = 0 \quad \text{ a.e. in } \ \Pi.
$$

Moreover, the continuity and linearity of the trace operators imply
that $w_v$ satisfies the corresponding scalar homogeneous boundary
conditions:
$$
\Gamma_I w_v=0,
\quad\Gamma_{J_0}(\partial_xw_v)=0,
\quad
\Gamma_{J_0}w_v=\Gamma_{J_{2\pi}}w_v.
$$
Thus for every $v\in \rX^*$
$ v(w)\equiv0 $ almost everywhere  in $\Pi.$

Since $w$ is strongly measurable, there exists a separable closed
subspace $\rX_0\subset\rX$ such that
$w(x,y)\in\rX_0$
for almost every $(x,y)\in\Pi$. Let
${v_j}_{j\in\mathbb N}\subset\rX^*$ be a countable norming family for $\rX_0$, so that
$$
\|z\|_{\rX} =
\sup_{j\in\mathbb N}|v_j(z)|,
\qquad z\in\rX_0.
$$
Since $v_j(w)=0$ almost everywhere in $\Pi$ for every
$j\in\mathbb N$, there exists a set of full measure in $\Pi$ on which
$v_j(w(x,y))=0$ 
for every $j\in\mathbb N.$

Hence, on this set,
$$
\|w(x,y)\|_{\rX}= \sup_{j\in\mathbb N}|v_j(w(x,y))| = 0.
$$
Thus
$w=0$ a.e. in $\Pi.$
Therefore, $u_1=u_2 $ almost everywhere in $\Pi.$

\end{proof}

The following result establishes the existence of a solution and provides
an a priori estimate.

\begin{thm}[Existence]\label{T3.1}
Let $\rX$ be a UMD Hilbert space and let $1<p<\infty$. Suppose that
$f\in W_p^2(I;\rX)$ satisfies
$
f(0)=f(2\pi)=0, $ $
f'(0)=0, $ 
and the compatibility condition
$$
\int_0^{2\pi}(2\pi-x)f(x)\,dx=0.
$$
Then the boundary value problem \eqref{BVP} admits a unique solution
$
u\in W_{p,1}^2(\Pi;\rX).
$
Moreover, there exists a constant $C>0$, independent of $f$, such that
$$
\|u\|_{W_{p,1}^2(\Pi;\rX)}
\leq C\|f\|_{W_p^2(I;\rX)}.
$$

\end{thm}

\begin{proof}
In analogy with the scalar case, we  consider the formal
series
\begin{equation}\label{E3.4}
u(x,y) = u_0+ \sum_{n=1}^{\infty}
\left( u_n(y)\cos(nx)+v_n(y)x\sin(nx)\right),
\qquad (x,y)\in\Pi,
\end{equation}
where the coefficients are determined by the   system
\eqref{eq-system2}. 
The boundary  conditions  and the compatibility condition give
\begin{equation}\label{eq-coef0}
\begin{split}
u_0&= t_0^c(f)= \frac{1}{2\pi^2}
\int_0^{2\pi}f(x)(2\pi-x)\,dx=0,\\
u_n(0)&=t_n^c(f) = \frac{1}{\pi^2}
\int_0^{2\pi}(2\pi-x)f(x)\cos(nx)\,dx=\frac{1}{\pi} \ell_n^c(g),\\
v_n(0) &=
-\frac{1}{2n}t_n^s(f) = \frac{1}{\pi^2} \int_0^{2\pi}f(x)\sin(nx)\,dx=\frac{1}{\pi}\ell_n^s(f).
\end{split}
\end{equation}

To have    
\begin{align*}
\Delta u &= \sum_{n=1}^{\infty}
\Big[
\left(u_n''(y)-n^2u_n(y)+2nv_n(y)\right)\cos(nx)\\
&+ \left(v_n''(y)-n^2v_n(y)\right)x\sin(nx)
\Big]=0
\end{align*}
we need that 
$$
v_n''(y)-n^2v_n(y)=0, \qquad u_n''(y)-n^2u_n(y)+2nv_n(y)=0.
$$
We select the solutions which decay as $y\to \infty.$  Hence by \eqref{eq-coef0}
\begin{align*}
v_n(y) &=v_n(0) e^{-ny}= \frac{1}{\pi} \ell_n^s(f) e^{-ny},\\
u_n(y) & =
\frac{e^{-ny}}{\pi^2}\left[ 
\int_0^{2\pi}(2\pi-x)f(x)\cos(nx)\,dx +y
\int_0^{2\pi}f(x)\sin(nx)\,dx\right].
\end{align*}

Thus
\begin{equation}\label{eq-series}
\begin{split}
u(x,y)=&\sum_{n=1}^{\infty}
\frac{1}{\pi^2}
\left(\int_0^{2\pi}(2\pi-\xi)f(\xi)\cos(n\xi)\,d\xi
\right) 
\cos(nx) e^{-ny}\\
&+\sum_{n=1}^{\infty}
\frac{y}{\pi^2}
\left( \int_0^{2\pi}f(\xi)\sin(n\xi)\,d\xi\right) \cos(nx) e^{-ny}
\\
&+\sum_{n=1}^{\infty}
\frac{1}{\pi^2}
\left(\int_0^{2\pi}f(\xi)\sin(n\xi)\,d\xi\right)  x\sin(nx) e^{-ny}.
\end{split}
\end{equation}
Direct calculations show that $\Delta u =0.$

We now  prove that the series in \eqref{eq-series}, together with its weak derivatives of order up to  two, converges in $L^{p,1}(\Pi;\rX).$ This will imply that  $u\in W^2_{p,1}(\Pi;\rX)$. 

We first consider
$$
u_1(x,y)=\sum_{n=1}^{\infty}v_n(y)x\sin(nx)=\frac{x}{\pi} \sum_{n=1}^{\infty} \ell_n^s(f) \sin (nx) e^{-ny}.
$$
Since $f \in W_p^2(I;\rX)$ with $1<p<\infty$, the one-dimensional
Sobolev embedding gives $ 
f,f'\in C(\overline{I};\rX).$ 
Integrating by parts and using the boundary conditions  we have
$ \ell_n^s(f)=-\frac{1}{n^2}\ell_n^s(f''). $
Therefore,
$$
u_1(x,y)
=-\frac{x}{\pi}\sum_{n=1}^{\infty}
\frac{1}{n^2}\ell_n^s(f'')\sin(nx)e^{-ny}.
$$
Direct calculations give 
\begin{align*}
\frac{\partial^2u_1}{\partial x^2}
&= \frac{x}{\pi}\sum_{n=1}^{\infty}  
\ell_n^s(f'')\sin(nx) e^{-ny}
-\frac{2}{\pi}\sum_{n=1}^{\infty}
\frac{\ell_n^s(f'')}{n}\cos(nx) e^{-ny},\\
\frac{\partial^2u_1}{\partial y^2}
&=-\frac{x}{\pi}\sum_{n=1}^{\infty} 
\ell_n^s(f'')\sin(nx) e^{-ny},\\
\Delta u_1&= -\frac{2}{\pi}\sum_{n=1}^{\infty}
\frac{\ell_n^s(f'')}{n}\cos(nx) e^{-ny}.
\end{align*}
Hence 
$$
\|u_1\|_{W^2_{p,1}(\Pi;\rX)}\leq C \|f''\|_{L^p(I;\rX)}\leq C\|f\|_{W^2_p(I;\rX)}.
$$

Consider now the term 
$$
u_2(x,y) = -
\sum_{n=1}^{\infty} n^2v_n(y)x\sin(nx)
= \frac{x}{\pi}
\sum_{n=1}^{\infty}
\ell_n^s(f'')\sin(nx)e^{-ny}.
$$

To estimate $\|u_2\|_{L^{p;1}(\Pi; \rX)}$, we distinguish two cases.

\textbf{I.} Suppose that $p > 2$. Then, $p' \in (1, 2)$. Since $\rX$ is a Hilbert space  the  $\rX$-valued \textit{Hausdorff--Young inequality} holds \cite{25}.  Hence
\begin{align*}
\left( \int_0^{2\pi} \| u_2(x,y) \|_{\rX}^p\, dx \right)^{1/p} &\leq C \left( \sum_{n=1}^{\infty} \|  \ell_n^s(f'') e^{-n y}\|_{\rX}^{p'}  \right)^{1/p'}\\
&\leq C \sum_{n=1}^{\infty} \|  \ell_n^s(f'') \|_{\rX}  e^{-ny} .
\end{align*}
In the last step, we used the imbedding $l^{p'}\hookrightarrow l^1$, $p'\geq 1$. Then
\begin{align*}
\| u_2 \|_{L^{p;1}(\Pi; \rX)} &\leq C \sum_{n=1}^{\infty} \|  \ell_n^s(f'') \|_{\rX} \int_0^{\infty} e^{-ny} \,dy
=C \sum_{n=1}^{\infty} \frac{\|  \ell_n^s(f'') \|_X}{n} \\
&\leq C \left( \sum_{n=1}^\infty \|\ell_n^s(f'')\|^{p'}_\rX  \right)^{1/p'} \left( \sum_{n=1}^\infty \frac{1}{n^p} \right)^{1/p} \\ &\leq C\|f''\|_{L^p(I;\rX)}\leq C \|f\|_{W^2_p(I;\rX)}.
\end{align*}


\textbf{II.} Suppose that  $1 < p \leq 2$. Then $p'\geq 2.$ Since $I$ has a finite measure, we have
$$
\left( \int_0^{2\pi} \| u_2(x,y) \|_{\rX}^p\, dx \right)^{1/p} \leq C \left( \int_0^{2\pi} \| u_2(x,y) \|_{\rX}^{p{'}} \, dx \right)^{1/p{'}}.
$$

Applying again the  $\rX$-valued Hausdorff--Young inequality, we obtain
\begin{align*}
\left( \int_0^{2\pi} \| u_2(x,y) \|_{\rX}^p\, dx \right)^{1/p} &\leq C \left( \sum_{n=1}^{\infty} \| \ell_n^s(f'') e^{-ny}\|_X^p  \right)^{1/p}\\
&\leq C \sum_{n=1}^{\infty} \|  \ell_n^s(f'') \|_X e^{-ny}.
\end{align*}
The desired estimate then follows as in Case~\textbf{I}.

Arguing analogously, we estimate the remaining terms in the representation of $u(x,y)$ obtaining
$$
\| u \|_{W_{p;1}^2(\Pi; \rX)} \leq C \| f \|_{W^2_p(I; \rX)}.
$$

We now verify   the boundary conditions in \eqref{BVP}. First, we show that $\Gamma_I u = f$. Since 
$$
\Gamma_{I} \in \left[ W_{p,1}^2(\Pi; \rX), W^1_p(I; \rX)\right],
$$
 it follows that $u_m \to u$  in $W_{p,1}^2(\Pi; \rX)$ implies  $\Gamma u_m \to \Gamma u$ in $W^1_p(I; \rX)$. 

Consider the partial sums
$$
u_m(x,y) =  \sum_{n=1}^m \left( u_n(y) \cos(nx) + v_n(y) x \sin(nx) \right), \quad   \text{ in } \Pi.
$$
Then 
$$
\Gamma_I u_m(x,y) =  \sum_{n=1}^m \left[  u_n(0) \cos(nx) + v_n(0) x \sin(nx) \right].
$$
If $u \in C^2(\overline{\Pi}; H)$, then $\Gamma_S u = u|_S$, for every half--line $S \subset \overline{\Pi}$. 
It is easy to verify  that 
$$
\left\{ u_n(y) \cos(nx), \ v_n(y) \,x \sin(nx) \right\} \subset C^2( \overline{\Pi}; \rX).
$$

Using \eqref{eq-series}, we obtain
\begin{equation}\label{E3.5}
\begin{split}
\Gamma_I u_m =&   \sum_{n=1}^m  \frac{1}{\pi^2} \left(\int_0^{2\pi} (2\pi - \xi) f(\xi) \cos(n\xi)\, d\xi\right) \cos(nx)\\
&+\frac{1}{\pi^2} \left(\int_0^{2\pi} f(\xi) \sin (n\xi) \,d\xi \right)  x\sin(nx), \qquad m\in \N.
\end{split}
\end{equation}

Since by Theorem \ref{T3.4}, the system \eqref{eq-system} forms a $\otimes$-basis in $L^p(I; \rX)$, it follows  that the right-hand side   of \eqref{E3.5} converges to $f$ in $L^p(I; \rX)$.

On the other hand,  $\Gamma_I u_m \to \Gamma_I u$ in $L^p(I; \rX)$ as $m\to \infty$.  Hence  by the uniqueness of the limit in $L^p(I;\rX)$ we have that  $\Gamma_I u = f$. 

The remaining boundary conditions  are verified analogously.

\end{proof}

\paragraph{\textbf{Acknowledgments.}} 
The research of B.~Bilalov and S.~Sadigova 
 is supported by the Azerbaijan Science Foundation – Grant No. AEF-BQM-Az-Turk-4-2025-6(59)-07/06/1-M-06.

P.~Salerno and L.~Softova are members of INDAM-GNAMPA.
The research of L.~Softova is partially supported by the FARB 300396FRB25SOFTO.

\paragraph{\textbf{Compliance with Ethical Standards}.}
The authors declare that they have no conflict of interest.

\paragraph{\textbf{Data Availability Statement}.}
Data sharing is not applicable to this article as no datasets were generated or analysed during the current study.

\paragraph{\textbf{ Author Contributions}.}
All authors contributed equally to the conception  of the study and 
the development of the mathematical arguments.
 All authors have read and approved the final version of the manuscript.

\end{document}